\documentclass{amsart}
\usepackage[utf8]{inputenc}
\usepackage{caption}
\usepackage{graphicx}

\title[Counting Rotational Sets for Laminations of the Unit Disk]{Counting Rotational Sets for Laminations of the Unit Disk from First Principles}
\author{John C. Mayer}
\author{Michael J. Moorman}
\author{Gabriel B. Quijano}
\author{Matthew C. Williams}
\date{September 2023}

\begin{document}

\newcommand{\zero}{\_0}
\newcommand{\ds}{\displaystyle}
\newcommand{\ra}{\rightarrow} \newcommand{\mc}{\mathcal}
\newcommand{\mrm}{\mathrm}
\newcommand{\lf}{\lfloor} \newcommand{\rf}{\rfloor}
\newcommand{\ul}{\underline} \newcommand{\ol}{\overline}
\def\interior{\mathrm{Int}} \def\plane{\mathbb{C}}
\def\sphere{\widehat{\mathbb{C}}} \def\nat{\mathbb{N}}
\def\reals{\mathbb{R}} \def\rat{\mathbb{Q}} \def\ints{\mathbb{Z}}
\def\disk{\mathbb{D}}
\newcommand{\rats}{\mathbb{Q}}
\newcommand{\zed}{\mathbb{Z}}
\newcommand{\complex}{\plane}
\newcommand{\wh}{\widehat}
\newcommand{\e}{\varepsilon}
\newcommand{\ucirc}{\mathbb{T}}
\newcommand{\ucurc}{\ucirc}
\newcommand{\0}{\emptyset}
\newcommand{\sse}{\subsection}
\newcommand{\ssse}{\subsubsection}
\newcommand{\op}{order-preserving\ }
\newcommand{\dist}{\mathrm{d}}
\newcommand{\sqr}{\mathrm{Sqr}}
\newcommand{\orb}{\mathrm{Orb}}
\newcommand{\sm}{\setminus}
\newcommand{\itin}{\mathrm{Itin}}
\newcommand{\Mod}[1]{\ (\mathrm{mod}\ #1)}

\newtheorem{thm}{Theorem}
\newtheorem{lem}[thm]{Lemma}
\newtheorem{cor}[thm]{Corollary}
\newtheorem{clm}[thm]{Lemma}
\newtheorem{prop}[thm]{Proposition}

\theoremstyle{definition}
\newtheorem{dfn}[thm]{Definition}
\newtheorem{ques}{Question}
\newtheorem{example}[thm]{Example}

\theoremstyle{remark}
\newtheorem{rem}[thm]{Remark}

\begin{abstract}
By studying laminations of the unit disk, we can gain insight into the 
 structure of Julia sets of polynomials and their dynamics in the complex plane. The polynomials of a given degree, $d$, have a parameter space. The hyperbolic components of such parameter spaces are in correspondence to rotational polygons, or classes of ``rotational sets'', which we study in this paper. By studying the count of such rotational sets, and therefore the underlying structure behind these rotational sets and polygons, we can gain insight into the interrelationship among hyperbolic components of the parameter space of these polynomials.

These rotational sets are created by uniting rotational orbits, as we define in this paper. The number of such sets for a given degree $d$, rotation number $\frac pq$, and cardinality $k$ can be determined by analyzing the potential placements of pre-images of zero on the unit circle with respect to the rotational set under the $d$-tupling map. We obtain a closed-form formula for the count.
Though this count is already known based upon some sophisticated results, our count is based upon elementary geometric and combinatorial principles,  and provides an intuitive explanation. \end{abstract}
\maketitle

\section{Introduction}

\subsection{Motivation}

What are ``rotational sets" for laminations for the unit disk under the action of the angle $d$-tupling map, and why count them? Laminations are a topological and combinatorial model of the connected Julia sets of polynomials considered as functions of the complex numbers, modeled by the plane.  Such models are used both to understand specific types of polynomials and their Julia sets, and to study the parameter spaces of polynomials.  For example, the well-known Mandelbrot set \cite{Mandel} is the parameter space of quadratic polynomials of the form $P_c(z)= z^2+c$ with parameter $c$ with connected Julia set.  The so-called hyperbolic components of that parameter space are of interest, including how they are connected to each other, how they are arranged in the Mandelbrot set, and how many components there are that are associated with attractive orbits (of the associated polynomials) of a given period, rotation number, and the like.  These terms are  defined below.  Our research is concerned with polynomials of higher degree ($d>2$), about which much less is currently understood.


Laminations are composed of {\em leaves} (chords of the unit circle) which form a closed collection of non-crossing segments that are forward and backward invariant under a natural extension of the {\em angle $d$-tupling map} (the angular part or argument of the complex power function $z\mapsto z^d $, where $z=r e^{2\pi it}$ and the argument is the exponent of $e$) on the unit circle.  Leaves connecting points of a rotational set in circular order form polygons in the lamination.  There is a correspondence between rotational polygons in laminations and fixed points of polynomials that have a non-zero infinitesimal rotation number (determined by the argument of the derivative of the polynomial at the fixed point). Such polygons are in correspondence to a fundamental class of hyperbolic components of the parameter space of degree $d$ polynomials with connected Julia set.

For example, in the  Mandelbrot set for the hyperbolic component marked star in Figure \ref{mandel}, all the Julia sets have a (repelling) fixed point (the marked point in the Julia set) which is represented in the lamination for that Julia set by a rotational triangle (marked star).  The Julia set is actually the boundary of the shaded blue region, which contains all the points running off to infinity under iteration of the polynomial.  The white regions in the lamination correspond to the black regions in  the ``filled-in" Julia set.  The filled-in Julia set consists of all points whose orbits under iteration of the polynomial are bounded.


``Counting ... from First Principles" in our title indicates that we will use the most fundamental geometric and combinatorial properties of the angle $d$-tupling map to make the count.
By studying laminations in the abstract without reference to a particular polynomial or Julia set, we aim to reverse the process by which a Julia set leads to a lamination. By understanding what is possible for laminations, we can constrain what is possible for locally connected Julia sets.  Our main result is Theorem \ref{count}.
For a preview of where that theorem takes us with Julia sets, skip ahead to view Figure~\ref{rabbits}.

\begin{figure}
    \includegraphics[width=.3\textwidth]{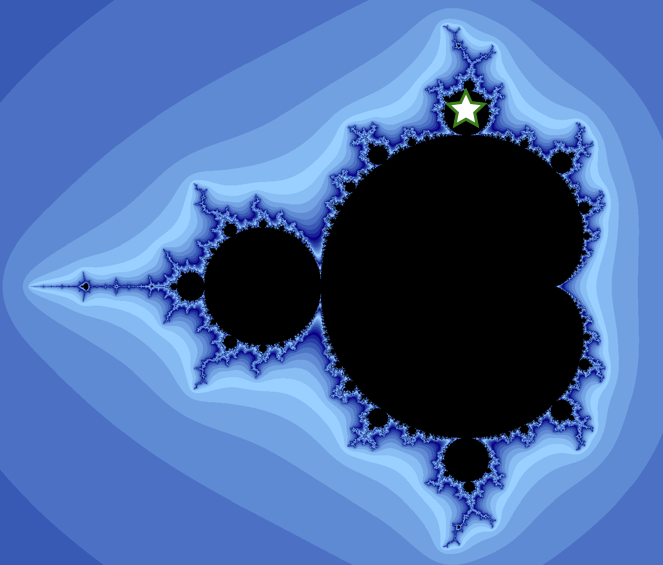}\hfill
    \includegraphics[width=.3\textwidth]{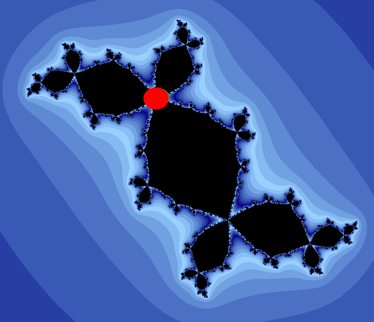}\hfill
    \includegraphics[width=.3\textwidth]{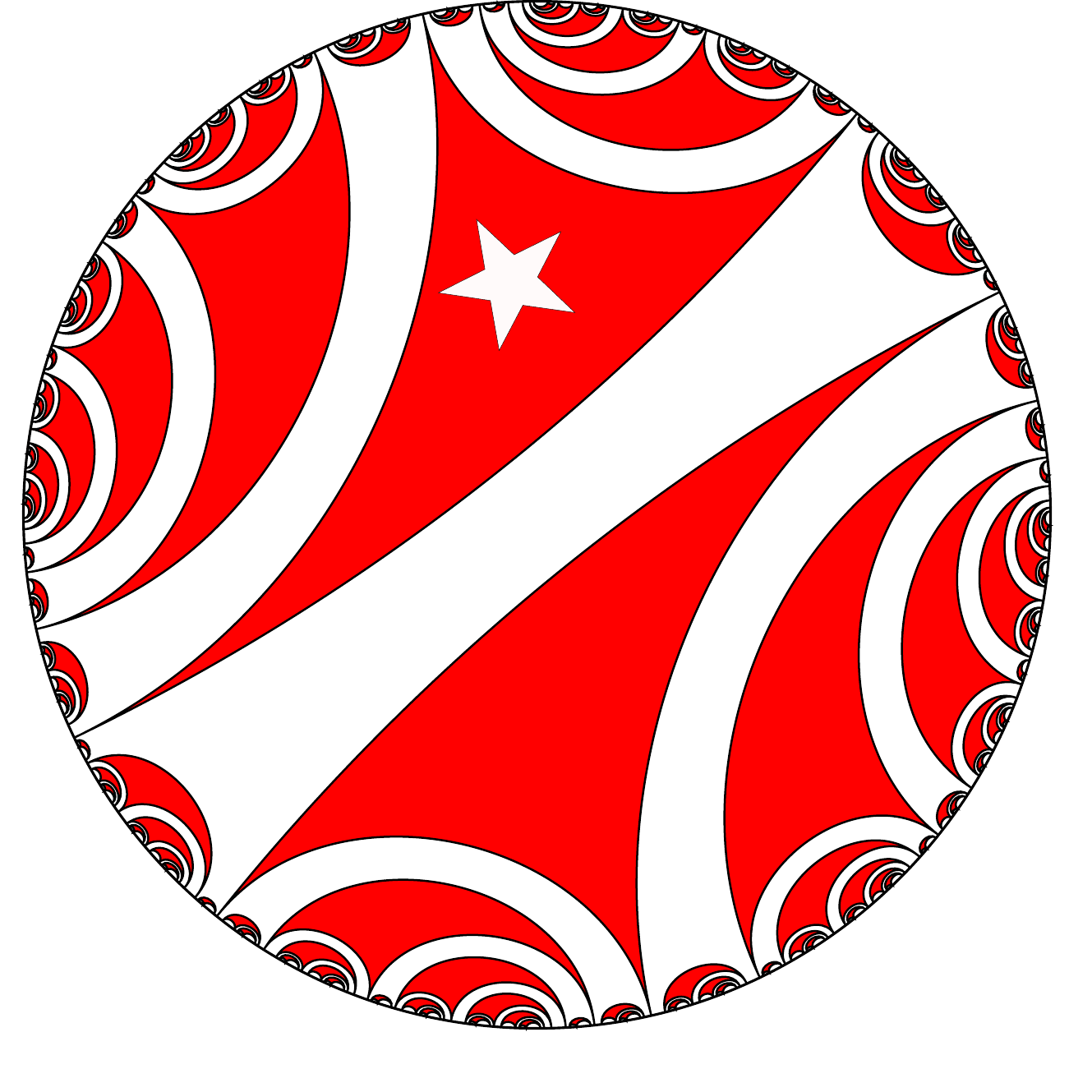}
    \caption{Hyperbolic component of the Mandelbrot set (marked star) with repelling fixed point (marked disk) in Julia set of $P(z)=z^2+(-0.117+0.743i)$ modeled by a rotational triangle (marked star) in the lamination for that Julia set. }
    \label{mandel}
\end{figure}

\subsection{Orbits, the Angle $d$-tupling Map, and Itineraries}
Most of the following definitions are adapted from   \cite{Brittany:2023} and \cite{Blokh:2006}. Elementary proofs of some propositions are left to the reader.


\begin{dfn}Let $f:X\to X$ be a function. By $f^q(x)$ we denote the composition $f(f(\dots f(x)\dots))$,  where $f$ is composed with itself $q$ times. 
\end{dfn}

By convention, $f^0(x)=x$.

\begin{dfn}
Given a point $x\in X$, the {\em orbit} of $x$ is the set $\mathcal{O} = \{f^n(x)\}^{\infty}_{n=0}$. 
\end{dfn}

We denote the unit circle in the Cartesian plane by
$S^1=\{(x,y)\mid x^2+y^2=1\}$.
We will
describe points in $S^1$ by their central angle and we will measure angles in revolutions rather than radians or degrees. A full circle is 1 revolution. We say we measure angles {\em mod 1}, such that angles which differ by a full revolution are the same.

We now consider the specific function $\sigma_d$ in whose orbits we are interested.

\begin{dfn}\label{d-map} Let $d\in\mathbb{Z}$ with $d>1$. Define the {\em angle $d$-tupling map} $\sigma_d:S^1\to S^1$
by $\sigma_d(x)=dx \pmod 1$. \end{dfn}

We represent points on the circle coordinatized by $[0,1)$ by their base $d$ expansion.  In base $d=2$, the notation $\_001$ denotes for us that the digits $001$ repeat infinitely which, as the reader can check, is the point $\frac13\in[0,1)$.  However, in base $d=3$, $\frac13$ is the point $1\_0$, which in our notation means that the initial digit $1$ does not repeat, but the digit $0$ repeats.  We can use the tools of symbolic dynamics, particularly the {\em forgetful shift}. Because points under $\sigma_d$ are multiplied by $d$ and are taken modulo $1$, the first digit of the base $d$ expansion becomes the integer part which goes away after taking the modulus. So we can quickly calculate the next point in an orbit by simply ``forgetting'' the leading digit (when written as a base $d$ expansion). 
For convenience, we will define $\mathbb{Z}_+$ to be $\{i\in\mathbb{Z}\mid i\geq0\}$.

We denote the set of {\em pre-images} under $\sigma_d^q$ of a point
$x$ by $\sigma_d^{-q}(x)=\{y\in S^1\mid \sigma_d^q(y)=x\}$. If there is a $q\in\mathbb{Z}_+$ such that $\sigma_d^q(x)=x$, then the orbit of
$x$ is finite, and we say it is a {\em periodic orbit} and $x$ is a
{\em periodic point}. If $q$ is least for which $\sigma_d^q(x)=x$, then we
say $q$ is the {\em period} of the point (respectively, orbit). The set of points visited on those $q$ iterations make up the orbit $\mathcal O$ for that given point. If this $q$ exists, then for the set ${\mathcal O} = \{\sigma_d^n(x)\mid 0 \le n < q\}$ it is true that $\sigma_d:\mathcal
O\to \mathcal O$ and $\sigma_d({\mathcal
O})={\mathcal O}$.


\begin{prop}
For a given degree $d$, the pre-images of $0$ are
\begin{equation}
 \sigma_d^{-1}(0)=\left\{0, \frac{1}{d}, \frac{2}{d},\dots,\frac{d-1}{d}\right\}
\end{equation}
or when written in base $d$ expansion  
\begin{equation}
    \sigma_d^{-1}(\_0)=\left\{\_0,1\_0,2\_0,\dots,(d-1)\_0\right\}
\end{equation}
\end{prop}

These pre-images serve as the border between neighboring intervals of length $\frac{1}{d}$ in $S^1$.

\begin{dfn}\label{preimage0} Fix $d>1$.  Define intervals $I_0=[0,\frac1d)$, and in general
$I_j=[\frac{j}{d},\frac{j+1}{d})$ for $1\le j\le d-1$.  Recall $0$
and $1$ are identical in $S^1$.  Then $S^1$ is the disjoint union $\bigcup_{j=0}^{d-1} I_j$.
\end{dfn}

For instance, the smallest non-zero pre-image of $0$, $\frac{1}{d}$ (equivalently, $1\_0$), is the border between $I_0$ and $I_1$.  Note that each interval $I_j$ maps one-to-one in counterclockwise order from $\_0$ onto $S^1$.  Consequently, within each $I_j$ there is a preimage of every $I_k$ consecutively in order, but of length $\frac{1}{d^2}$.
With these tools in hand, we can now express orbits in a much more useful manner.

\sse{Itineraries}  We have given the points in $S^1$ when considered under the map $\sigma_d$ their base $d$ expansion.
This will allow us to describe the orbit of a point in the circle under $\sigma_d$ in terms the visits of its orbit to the distinguished intervals in Definition~\ref{preimage0}.
The {\em itinerary} of a point $s\in S^1$ is the ordered list of its visits in its orbit to the distinguished intervals. 

The proof of the following Proposition is left to the reader.
\begin{prop}\label{itin}
    Under $\sigma_d$, the itinerary of a point is exactly its base $d$ expansion.  Consequently, two points $s$ and $t$ in $S^1$ under $\sigma_d$ have the same itinerary if, and only if, $s=t$.
\end{prop}

Now, the points of an orbit are defined by their relative placement to pre-images of zero. The utility in this is found in how orbits can be defined by the location of pre-images relative to the points of the orbit. This will be essential to counting rotational sets from first principles.

\begin{figure}[h]
    \includegraphics[scale=0.335]{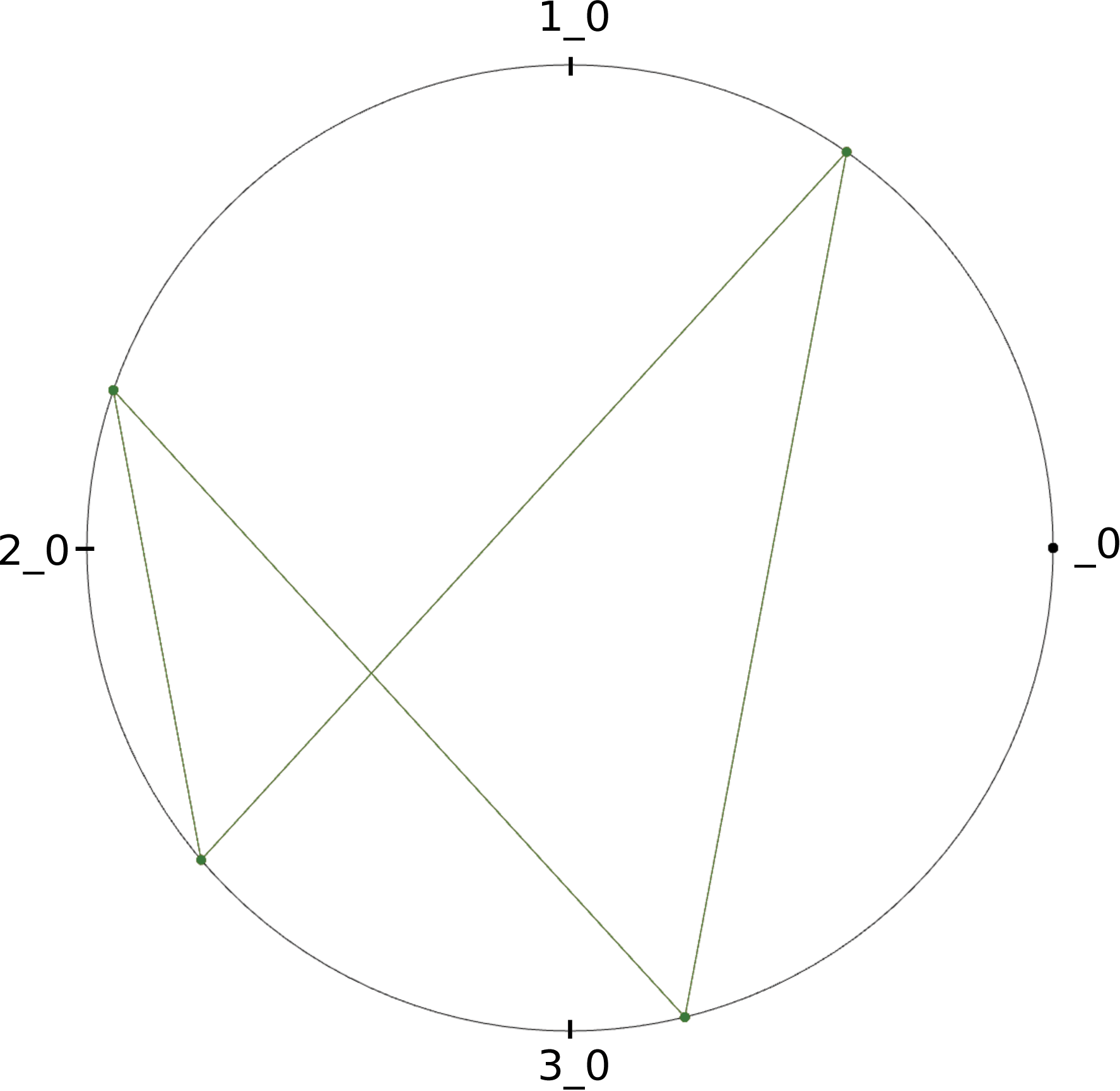}
    \caption{See if you can identify this orbit for $\sigma_4$ based on the position of its points relative to pre-images of $\zero$.}
    \label{fig:itinerary}
\end{figure}

Because any given itinerary for a point in a periodic orbit can be shifted to find the itineraries of every other point in that orbit, periodic orbits can be clearly referenced with the itinerary of any given point it contains. For consistency we will define orbits to have the same itinerary as their smallest (compared to $\_0$) point.

\begin{dfn} Let ${\mathcal O}=\{0\leq x_1<x_2<x_3\dots<x_q<1\}$ be a periodic orbit. Define $Itin(\mathcal{O})=Itin(x_1)$.
\end{dfn}

\begin{dfn}
A {\em gap} is a complementary interval in $S^1 \setminus \mathcal{O}$.
\end{dfn}

\sse{Spatial and Temporal Order of Orbits} 

The count of rotational sets is fundamentally based upon the comparison of spatial and temporal orders of points in an orbit with reference to the pre-images of $\_0$.  See Figure \ref{fig:temporal-spatial}.

\begin{dfn}
{\em Spatial order} refers to the ordering of points in an orbit by value, from smallest to largest in $[0,1)$. In terms of a map onto $S^1$, this order is given by starting at $\zero$ and following our points counterclockwise. 
\end{dfn}

\begin{dfn}
{\em Temporal order} orders the points in an orbit by starting with the lowest valued point in our orbit (or closest to $\zero$ spatially, measuring from 
$\zero$ counterclockwise) and applying $\sigma_d$ repeatedly. So, temporal ordering is based on how our itinerary is followed by repeated applications of $\sigma_d$. 
\end{dfn}

\begin{figure}[h]
    \captionsetup{justification=centering}
    \includegraphics[scale=0.28]{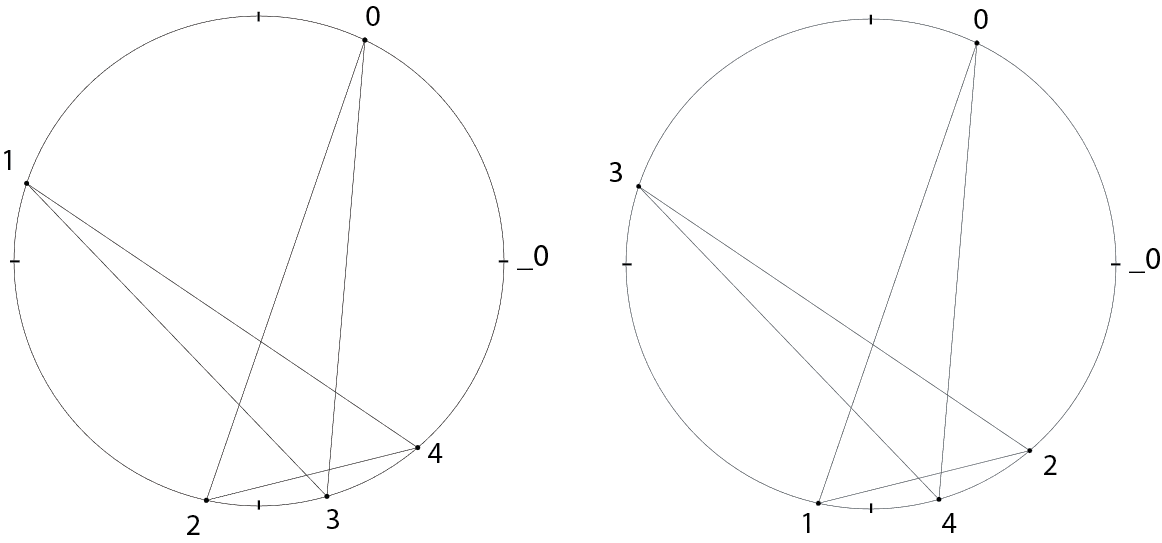}
    \caption{Spatial ordering left; Temporal ordering right.}
    \label{fig:temporal-spatial}
\end{figure}





\section{Rotational Sets}


Consider $\sigma_d:S^1\to S^1$ for a particular $d > 1$.  

\begin{dfn}
Let ${\mathcal O}=\{x_1<x_2<x_3\dots<x_q\}$ be a periodic orbit under $\sigma_d$. If, and only if, there exists a $p \in \mathbb{Z}^+$ such that for all $i \in \{1,2,\dots q\}$, $\sigma_d(x_i)= x_{i+p\pmod q}$, then we say that $\mathcal{O}$ is a {\em rotational} periodic orbit with {\em rotation number} $\frac{p}{q}$ (in lowest terms).
\end{dfn}

\begin{rem}\label{itinToPreimage}
Note that the numerator $p$ of the rotation number $p/q$ of a rotational periodic orbit is sufficient to determine the temporal order of its points. Along with how many points of the orbit are in each interval, this is enough to determine its itinerary.

It follows from Proposition \ref{itin} that if, and only if, two rotational periodic orbits have the same $p/q$ and each of their corresponding points are in the same intervals, then they have the same itinerary and are the same orbit.

As for the practical generation of this itinerary, it can be found by reading off the interval of each point by starting with $x_1$ and ``jumping'' forwards $p$ points counter-clockwise along $S^1$, repeating until getting back to that initial point.  Consider what $p$ is in the rotational orbit in Figure \ref{fig:temporal-spatial}.
\end{rem}


A consequence of how the rotation number of an orbit can describe the forward orbits of its points is that it can also describe their pre-images. Since $\zero$ lies between $x_1$ and $x_q$, at least one pre-image of $\zero$ must lie between $x_{q-p} \in \sigma_d^{-1}(x_{q})$ and $x_{1+q-p} \in \sigma_d^{-1}(x_{1})$. We call such pre-images the {\em Principal Pre-image} of their respective orbits.

\begin{dfn} Let the {\em principal preimage} for a rotational orbit $\mathcal{O}$ be a pre-image of $\zero$ lying between $x_{q-p} \in \sigma_d^{-1}(x_{q})$ and $x_{1+q-p} \in \sigma_d^{-1}(x_{1})$. 
\end{dfn}

The reader can verify that there must always be a principal preimage.

\sse{Rotational Sets Containing Multiple Orbits}
Not only can points rotate together while maintaining order, but so too can multiple orbits together, forming a rotational set.

\begin{dfn}Let ${\mathcal
P}=\{x_i\mid 0\leq x_1<x_2<x_3<\dots<x_{qk}<1\}$ be a finite set in consecutive counterclockwise order
in $S^1$.  We say $\mathcal P$ is a {\em rotational set} containing $k$ orbits with rotation number $\frac{p}{q}$ for
$\sigma_d$ if, and only if, 
\begin{enumerate} 
    \item $\sigma_d({\mathcal{P}})={\mathcal P}$ 
    \item $x_i,x_{i+1},\dots,x_{i+k-1 \pmod {qk}}$ for $i \in [1,2,\dots qk]$ are in different orbits 
    \item $x_i$ moves to $x_{i + pk\pmod{qk}}$ 
\end{enumerate}
\end{dfn}

\begin{dfn} Let $G_i = \{x_{(i-1)k + 1}, x_{(i-1)k + 2},\dots,x_{ik} \}$ where $ i\in[1,2,\dots,q]$. We say $G_i$ is the $i$th {\em group} and $G_i$ is the set of the $i$th points spatially of each orbit.
\end{dfn}

\begin{rem}
Items (2) and (3) of Definition 2.3 also show that $G_i$ moves together preserving spatial order to $G_{i + p\pmod{q}}$.
\end{rem}

\begin{dfn}  Let a {\em principal preimage} for a rotational set ${\mathcal P}$ be the pre-image of $\zero$  that lies between $G_{q-p} \subset \sigma_d^{-1}(G_{q})$ and $G_{q-p+1} \subset \sigma_d^{-1}(G_{1})$.
\end{dfn}

\begin{rem} \label{principal} For the principal pre-image of a set, it must lie between $x_{(q-p)k} \in \sigma_d^{-1}(x_{qk})$ and $x_{1+(q-p)k} \in \sigma_d^{-1}(x_{1})$. $x_{1+(q-p)k}$ is the smallest point in group $G_{q-p+1}$ and $x_{(q-p)k}$ is the largest point in group $G_{q-p}$, therefore the principal pre-image lies between those two groups.
\end{rem}

\begin{figure}[h]
    \centering
    \includegraphics[scale=0.35]{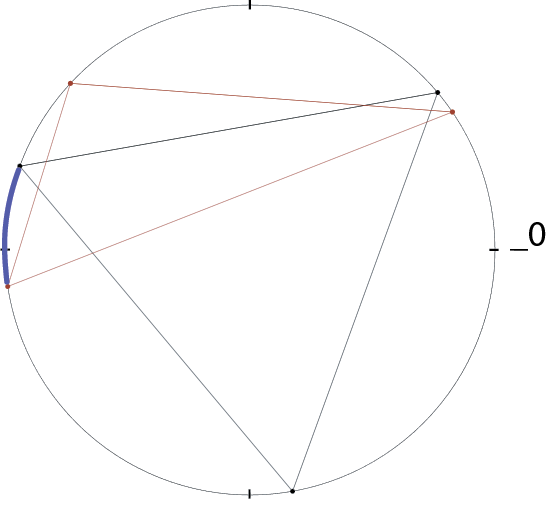}
    \caption{These two rotational orbits for $\sigma_4$ form a rotational set. The high-lighted inter-group gap contains the principal pre-image of this rotational set. See if you can identify these orbits, as practice.}
    \label{fig:rotSet}
\end{figure}

In order to aid with the identification and counting of rotational sets, we need to differentiate gaps between points that are within groups from those outside groups. This distinction is necessary as pre-images that lie in the former are what differentiate orbits within a rotational set from each other. In other words, the pre-images that are within the groups are such that if they were removed, the orbits would no longer be different.


\begin{dfn} Let {\em intra-group} gaps be gaps that lie within a group, or in other words are gaps that are in between two points from two different orbits that aren't the first and last orbits spatially in their group.
\end{dfn}
 
\begin{dfn} Let {\em inter-group} gaps be gaps that lie outside a group (or between groups). These are all of the gaps that are not intra-group gaps. 
\end{dfn}

Note that with these definitions, the principal pre-image must always lie within an inter-group gap. See if you can build intuition for this by considering Figure \ref{fig:rotSet} through this new lens.

\section{Algorithms and resulting Formulas} 

\subsection{Counting Rotational Sets}

 Goldberg \cite{Goldberg:1992} counted the rotational orbits with a given rotation number $\frac{p}{q}$ for $\sigma_d$ and indicated that rotational sets containing multiple orbits with that rotation number could also be counted, providing an example count for $\sigma _3$, but no general formula.  As a corollary to her characterization of rotational orbits in  terms of their temporal and spatial placement with respect to the $d-1$ fixed points of $\sigma_d$, she showed that the maximal number of orbits in a rotational set for $\sigma_d$ was $d-1$.  This also follows as a corollary to our main theorem, Theorem \ref{count}.  McMullen (\cite{McMullen:2010}, Section 2) built upon Goldberg to provide a criterion for two orbits for $\sigma_d$ with the same rotation number to be compatible in one rotational set. Tan \cite{Tan:2019} used an algorithm based upon the Goldberg/McMullen criterion to count the number of rotational sets containing $k$ orbits for $\sigma_d$ with a given rotation number $\frac pq$. So, while the content of our main theorem is known, the proof here is new and more elementary.


\begin{thm}[Identifying and Counting Rotational Sets]\label{count} Consider the collection $\mathcal{B}$ of all rotational sets for a given degree $d$, rotation number $\frac{p}{q}$ in lowest terms, and number $k$ of distinct orbits per set. 
The cardinality of $\mathcal{B}$ is given by 
\begin{equation}
|\mathcal{B}| =  \sum_{i=k-1}^l \left[ \binom{d+q-2}{d-2-i} \sum_{j=0}^{k-1} \left[ (-1)^j \binom{k-1}{j} \binom{q(k-1-j)}{i} \right] \right]
\end{equation}

\noindent where

\begin{equation}
l = 
\begin{cases} 
      q(k-1) & d-2 > q(k-1) \\
      d-2 & otherwise
   \end{cases}
\end{equation}

\end{thm}
\begin{proof}
For any given rotational set $B$ within $\mathcal{B}$, there are $q$ points in each of the $k$ orbits within $B$. The count of each way to place pre-images between these neighboring points is the same as the count of rotational sets because the placement of pre-images dictates the digits for the itinerary of each point as can be seen in Remark \ref{itinToPreimage}. However, placing a pre-image between $\zero$ and the smallest point in the set is different from placing a pre-image between the largest point and $\zero$. The former would increase the digits in all the itineraries while the latter would not. Therefore, we need to count the number of ways to place pre-images within the gaps distinguished by $\zero$ and the points within $B$.


Let the range of values from $0$ to $qk$ correspond with the gaps between neighboring points in the set of points that contain $\zero$ and the points within each orbit in $B$. $0$ corresponds with the gap between $\zero$ and the smallest point in $B$, $1$ with the gap between the smallest and second smallest, and so on. 

There are $d$ pre-images to place. The first is $\zero$, which is its own pre-image. The next is the principal pre-image, whose position is already determined (as can be seen in Remark \ref{principal}). This leaves us with $d-2$ pre-images to place.

Now, we must concern ourselves with where these pre-images can be placed to form a valid rotational set with $k$ orbits.

\begin{lem} Label gaps with their congruence class modulo $k$.
Rotational sets in $\mathcal{B}$ must have at least one pre-image in each  non-zero congruence class. Therefore, the cardinality of $\mathcal{B}$ is equivalent to that of $P$ when $P$ is defined to be the set of all sets composed of $d-2$ non-negative integers less than or equal to $qk$, such that each one contains at least one element from each non-zero congruence class modulo $k$.
\end{lem}

\begin{proof}

 In order to ensure differentiation between orbits, each orbit must be differentiated from its intra-group neighbors (the groups are differentiated by the principal pre-image and $\zero$). The first spatial orbit is not a neighbor with the last as they lie on opposite sides of any given group. So for each intra-group neighbor, an orbit must have at least one pre-image between one of its points and that neighbor's points. Therefore, this requirement can only be fulfilled by placing pre-images in intra-group gaps. 

Here is an example to provide clarity.  The reader is invited to draw their own illustration for this example. For the first orbit (spatially) to be distinct from the second orbit, there must be a pre-image either in the gap between their smallest points respectively, second smallest, or any other pair of corresponding points. This particular restriction for the first and second  orbits can be restated as the requirement for a pre-image to exist in a gap with label $n$ such that $n$ is in the congruence class of $1$ modulo $k$. This rule can be generalized for all intra-group gaps by separating them into congruence classes. For the $n$th and $(n+1)$th orbits to be differentiated, there must be at least one pre-image in the set of gaps with labels in $\{ x \in [0\dots qk] \mid x \in [n]_k\}$. Therefore, all non-zero congruence classes require at least one pre-image for the rotational set to be valid.


The set of rotational sets, with the restrictions articulated above, will have the same cardinality as $P$ when defined as follows. $P$ is the set of all sets composed of $d-2$ non-negative integers less than or equal to $qk$, such that each one contains at least one element from each non-zero congruence class modulo $k$.

\end{proof}

We will now define sets that correspond with the choice of gaps in which to put pre-images. As of now, we are only placing one pre-image per gap even though more than one can be placed for a valid rotational set. This is done to make counting simpler later on.

We will refer to the range of labels for gaps as $\lambda$, and define it as: 
\begin{equation}
    \lambda = \{x \in \mathbb{Z}_{+} \mid x \leq qk\}
\end{equation}
In other words, $\lambda$ is the set of non-negative integers less than or equal to $qk$. We will refer to the set of all labels for intra-group gaps as $\psi$ and define $\psi$ as 
\begin{equation}
    \psi = \{x \in \lambda \mid x \Mod{k} \neq 0 \}
\end{equation}

There are multiple possible values for how many pre-images may be in intra-group gaps for any given rotational set. We will call the number of pre-images in intra-group gaps $i$. 

Define $T_i$ as the set of all sets that contain $i$ elements from $\psi$. We also know that $|T_i| = \binom{|\psi|}{i} = \binom{q(k-1)}{i}$ since $|\psi| = q(k-1)$ because there are $k-1$ non-zero congruence classes with $q$ integers in each.

\begin{equation}
    |T_i| = \binom{q(k-1)}{i}
\end{equation}

\begin{lem}
    The range for the possible number of pre-images in intra-group gaps varies from $k-1$ to $l$ where
\begin{equation}
l =
    \begin{cases} 
      q(k-1) & d-2 > q(k-1) \\
      d-2 & otherwise
   \end{cases}
\end{equation}
\end{lem}

\begin{proof}

 Only $k - 1$ pre-images are required for differentiation. The minimal value of $i$ is the minimal number of pre-images required for differentiation, $k - 1$. As for the maximum value of $i$, it can be limited by either the number of empty gaps, $q(k - 1)$, or the number of pre-images to place, $d-2$. Therefore, the maximal value of $i$ is $l$ where

\begin{equation}
l = 
\begin{cases} 
      q(k-1) & d-2 > q(k-1) \\
      d-2 & otherwise
   \end{cases}
\end{equation}

\end{proof}

The sets in $T_i$ for $i \in [k-1, k, \dots ,l]$ that follow our requirement of differentiation are valid (yet may not be complete as these only correspond to $i$ pre-images of the $d-2$ that need to be placed). So we will define $\gamma_i$ as the subset of $T_i$ that contains all the sets where there is at least one element from each non-zero congruence class modulo $k$.


\begin{lem}
The count of sets in $\mathcal{B}$ such that each $B \in \mathcal{B}$ has $i$ pre-images in its intra-group gaps is given by
\begin{equation}
\sum_{j=0}^{k-1} (-1)^{j} \binom{k-1}{j} \binom{q(k-1-j)}{i}
\end{equation}
\end{lem}

\begin{proof}

Each set within $P$ that has $i$ elements in non-zero congruence classes corresponds with an element in $\gamma_i$. Similarly, each rotational set in $\mathcal{B}$ that has $i$ pre-images in intra-group gaps corresponds with an element in $\gamma_i$. For counting purposes, we found it easier to find $T_i \setminus \gamma_i$ than $\gamma_i$, so we will define $W_i=T_i \setminus \gamma_i$.

We can identify and count the sets in $W_i$ (as opposed to the valid sets in $\gamma$) by noticing the equivalence of the problem with finding the union of sets. The goal is to find every possible way to leave at least one congruence class unfilled. Categorize each placement as belonging to a series of sets $(C_1, C_2,\dots , C_{k-1})$ where $C_j$ is defined as the set of elements of $T_i$ where the $j$th congruence class is not represented. With this definition, a singular placement can belong to multiple sets $C_j$, and we seek the union of each of these sets. In other words, $\bigcup\limits_{j=1}^{k-1} C_j = W_i$ because it contains every set where at least 1 non-zero congruence class is not represented. This is a well known problem, and the solution utilizes the inclusion-exclusion  principle \cite{InEx} with the count given by


\begin{equation} 
|W_i| = \left|\bigcup\limits_{j=1}^{k-1} C_j\right| =  \sum_{j=1}^{k-1}  (-1)^{j+1} \binom{k-1}{j} \binom{q(k-1-j)}{i}
\end{equation}
Since $\gamma_i$ is defined as the set difference of $T_i$ and $W_i$, 
\begin{equation}
|\gamma_i| = |T_i| - |W_i| = \binom{q(k-1)}{i} - \sum_{j=1}^{k-1} (-1)^{j+1} \binom{k-1}{j} \binom{q(k-1-j)}{i}    
\end{equation}which can be rewritten as \begin{equation}
\sum_{j=0}^{k-1} (-1)^{j} \binom{k-1}{j} \binom{q(k-1-j)}{i}
\end{equation}

\end{proof}

We have identified and counted all valid placements of intra-group pre-images. In order to finish constructing a rotational set, the remaining pre-images must be placed in between groups. For each of these placements in $\gamma$, the number of pre-images left to place naturally depends on the number of pre-images already placed. For any given set $g \in \gamma$ with cardinality $i$, there are $d-2-i$ elements from $\lambda$ (all the labels for the gaps) that need to be added to construct an element in $\mathcal{B}$. 

There are a few limitations on which gaps pre-images can be placed in, due to the fact that our prior count depends on certain gaps lacking pre-images and others having at least one. Therefore, we can place pre-images in intra-group gaps that already have pre-images or any given inter-group gap. In other words, the added $d-2-i$ elements must either be duplicates of elements already in $g$ which are also in $\psi$ or elements in $\lambda \setminus \psi$. There are $\binom{d+q-2}{d-2-i}$ ways to choose these integers, therefore there are these many elements in $P$ and corresponding orbits for any given set $g$.

As argued before, $i$ (the cardinality of $g$) ranges from $k-1$ to $l$.   


Bringing all of this together, $|\mathcal{B}|$ and $|P|$ are equal to the sum of $\binom{d+q-2}{d-2-i}|\gamma|$ over possible values of $i$.


\begin{equation}
|P| = |\mathcal{B}| = \sum_{i=k-1}^l \left[ \binom{d+q-2}{d-2-i} \sum_{j=0}^{k-1} \left[ (-1)^j \binom{k-1}{j} \binom{q(k-1-j)}{i} \right] \right]
\end{equation}
\end{proof}

\subsection{Identifying Rotational Sets Containing a Given Orbit}

This count is certainly insightful, but in order to gain insight as to specific laminations that one may examine, it would be useful to be able to count and identify rotational sets that contain a given orbit. Here we state that this can be done, using a method quite similar to the method used in the proof of the previous counting theorem.

\begin{thm}[Identifying Maximal Rotational Sets that contain a given orbit] Consider the orbit $\mathcal{O}$ with degree $d$ and rotation number $\frac{p}{q}$ in lowest terms. An exhaustive list of all the rotational sets that contain $\mathcal{O}$ can be found algorithmically, in a way inspired by the previous proof. The count of maximal rotational sets that contain $\mathcal{O}$ can be expressed as a closed-form formula in terms of the degree $d$, the rotation number $\frac{p}{q}$, and the digits of $\mathcal{O}$.
\end{thm}

Similar to the count of all rotational sets under certain parameters being given through the valid placements of pre-images of $\zero$, the orbits that belong to rotational sets containing a given orbit $\mathcal{O}$ can be identified through a similar process.
We leave this investigation to the reader. However, an algorithm for identifying such rotational sets containing an orbit $\mathcal{O}$ based on the proof strategy of the previous theorem can be accessed on GitHub \cite{Hugh:2023}.

\subsection{Examples}  In order to demonstrate the usefulness of the proposed algorithm, consider the rotational orbit $[\_012, \_120, \_201]$ under $\sigma_3$.  Applying the  algorithm, we determine that this orbit can be paired with either the orbit $[\_002, \_020, \_200]$ or the orbit $[\_112, \_121, \_211]$ (but not both since a rotational set for $\sigma_3$ contains at most two orbits).  The rotational polygons formed by these three sets are in the first row of Figure~\ref{rabbits}.

\begin{figure}
    \centering
    \includegraphics[width=.33\textwidth]{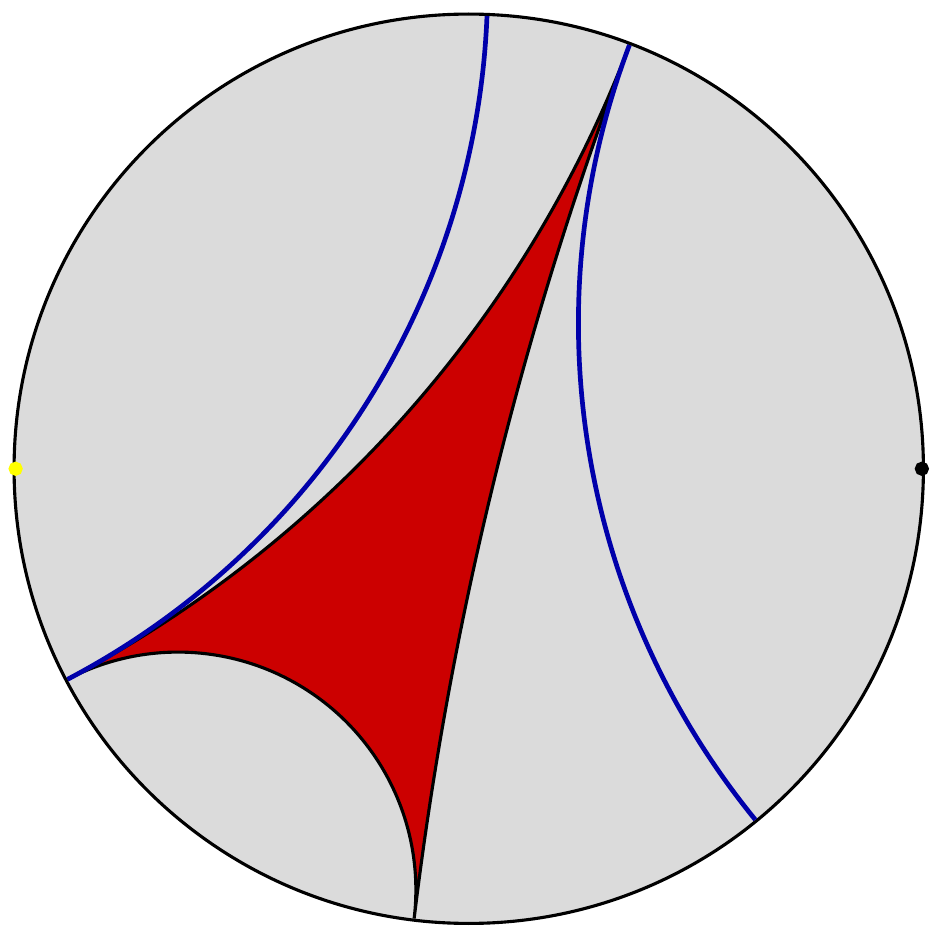}\hfill
    \includegraphics[width=.33\textwidth]{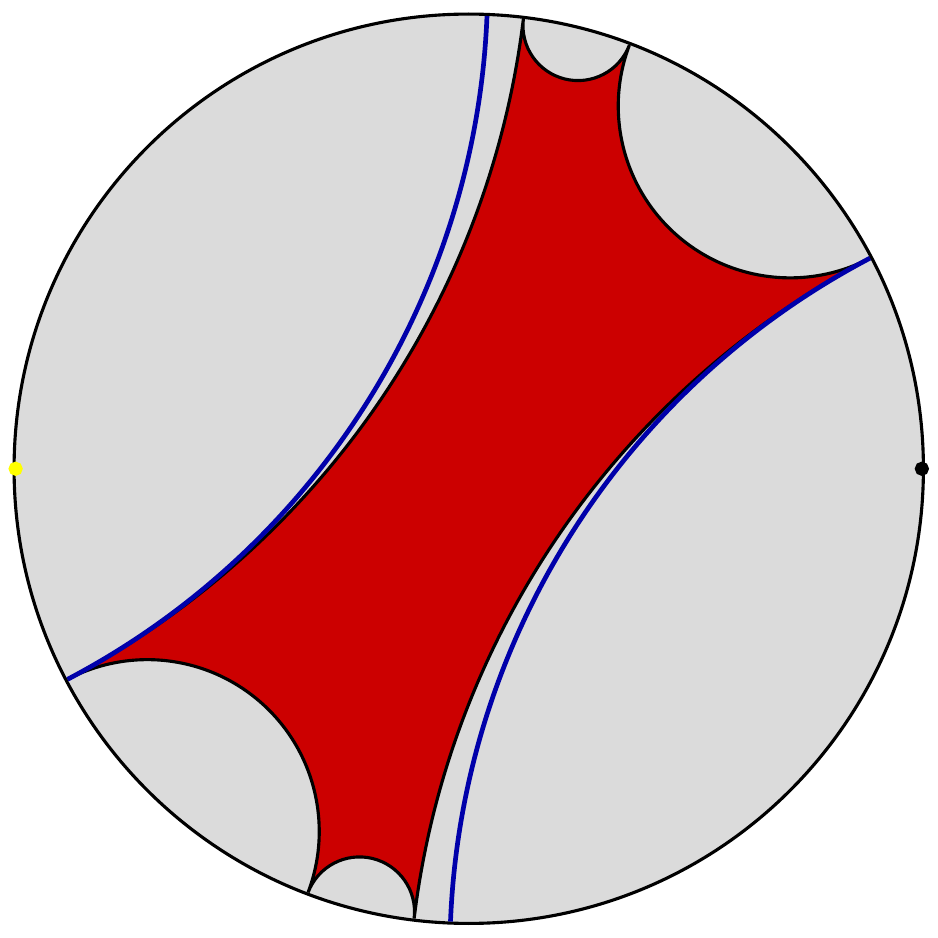}\hfill
    \includegraphics[width=.33\textwidth]{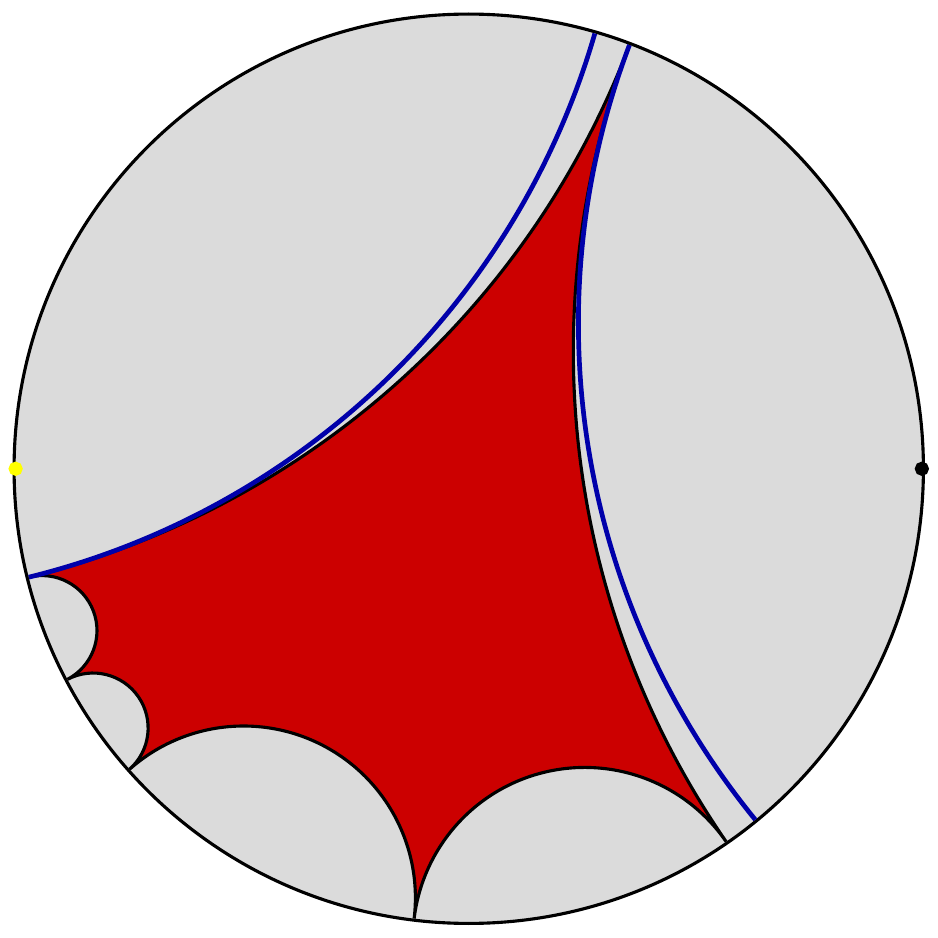}\\
     \includegraphics[width=.33\textwidth]{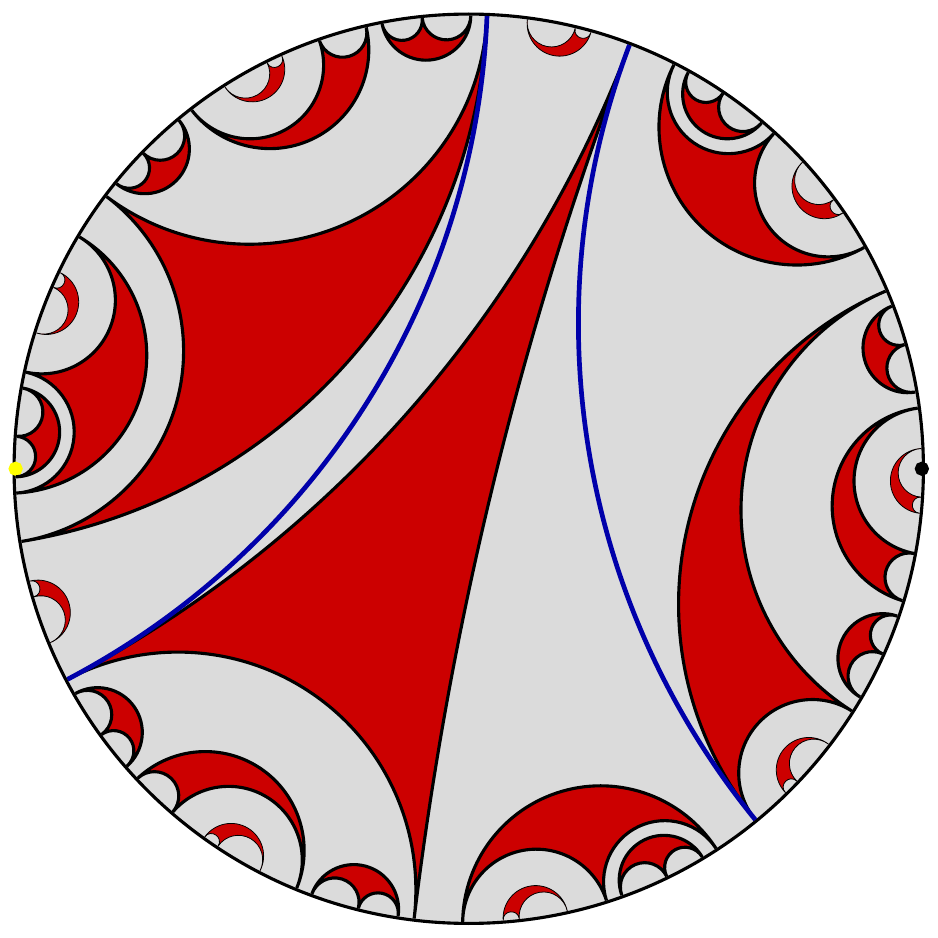}\hfill
    \includegraphics[width=.33\textwidth]{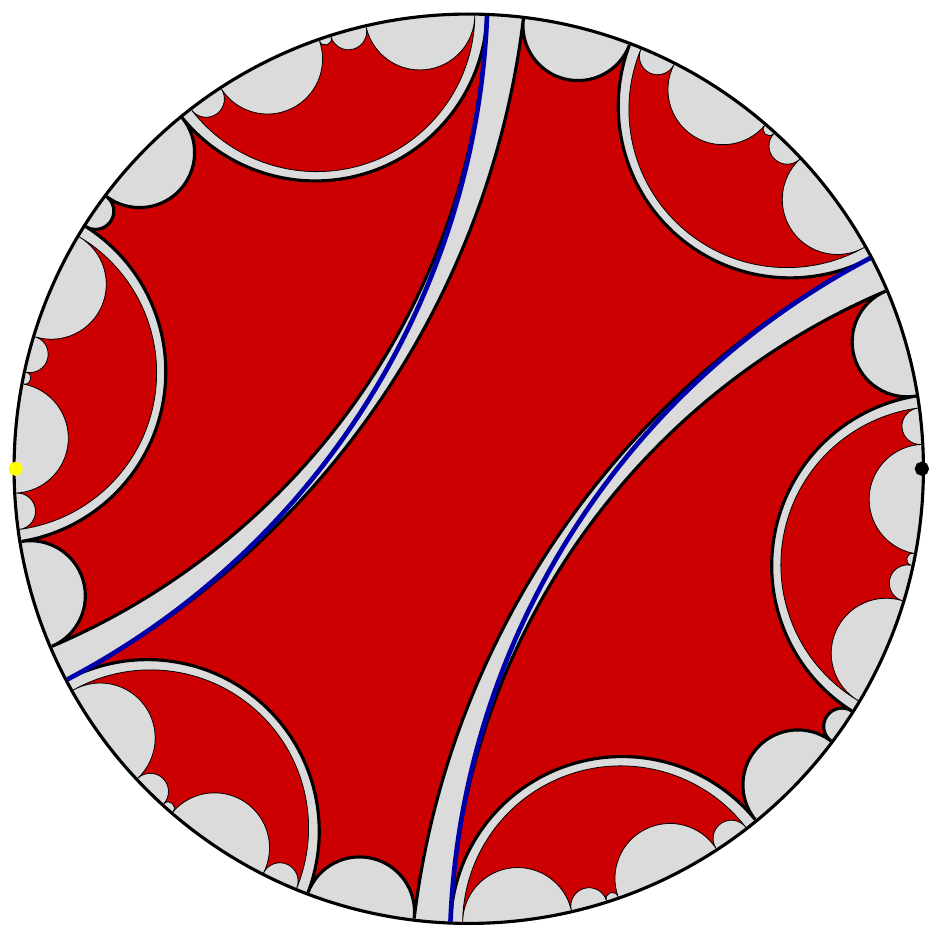}\hfill
    \includegraphics[width=.33\textwidth]{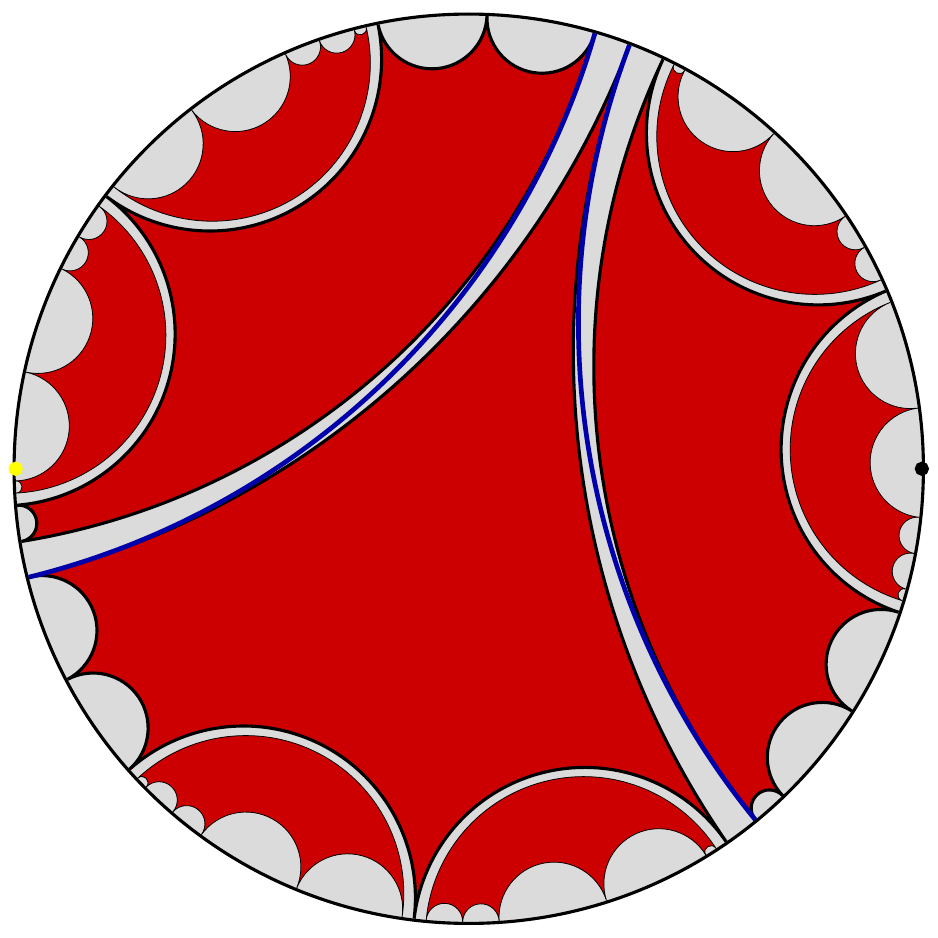}\\
     \includegraphics[width=.33\textwidth]{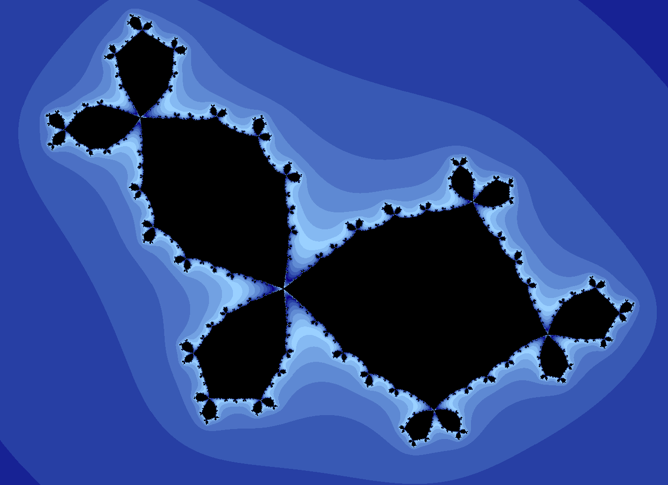}\hfill
    \includegraphics[width=.33\textwidth]{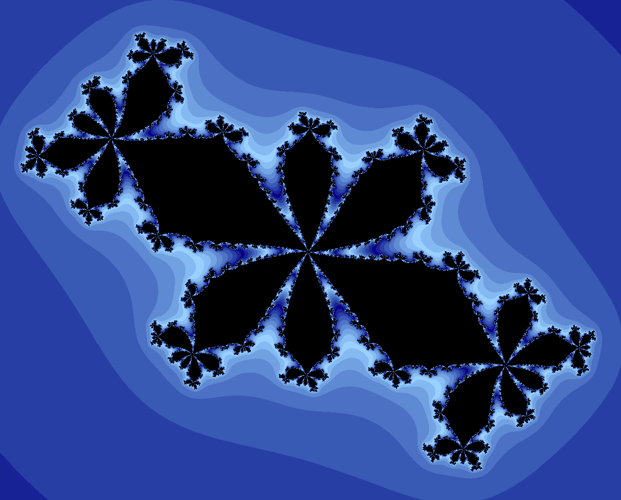}\hfill
    \includegraphics[width=.33\textwidth]{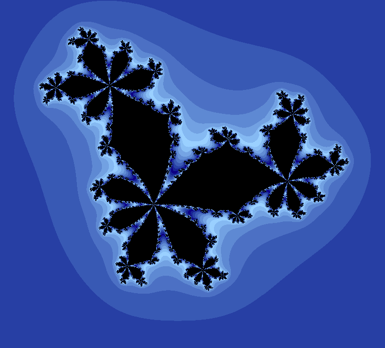}
    \caption{Triangle and two hexagon expansions for $\sigma_3$.\\
    Initial rotational data (left to right):\\
    ${\_012}$ orbit, ${\_012}$ \& ${\_002}$ orbits, ${\_012}$ \& ${\_112}$ orbits.\\
    Pullbacks laminations (left to right): \newline 3 steps, 2 steps, 2 steps.\newline
    Polynomial equations (left to right): \newline $P(z)=(-0.607786+1.05435i)z-(0.316863+0.325784i)z^2+z^3/3$ \newline
    $P(z)=(0.89012-0.512216i)z^2+z^3/3$\newline
    $P(z)=-0.156693-0.682581i+(-0.355147+0.182883i)z^2+z^3/3$}
    \label{rabbits}
\end{figure}

The second row of the figure shows the first few stages of the pullback lamination determined by the initial polygons and an appropriate choice of branches of the inverse of $\sigma_3$ \cite{LamBuilder:2021}.  The corresponding Julia sets, found using Mathematica and FractalStream \cite{Brittany:2023}, are displayed in the third row of the figure.  Each polygon in the corresponding lamination represents a  junction point in the Julia set.  The central polygon in each case represents a fixed point of the polynomial for the corresponding Julia set.  Note that the lamination represents reasonably well the geometry of the corresponding Julia set if you imagine the polygons shrunk to points.  In this process it is important to note that we found the rotational polygons and laminations {\it before} we found the polynomials and their Julia sets. 




\subsection{Future Questions}
If you find yourself interested in continuing this work, perhaps consider the following questions as starting points for areas of research:
\begin{enumerate}
    \item The formula for the count provided in Theorem \ref{count} is quite complicated. Tan (\cite{Tan:2019} Theorem 3.2) provides an equivalent count, though without a basis in first principles. How can our count be simplified from first principles?
      \item Verify that the formula  found by Tan and the closed-form formula in our Theorem 3.1 give the same count of rotational sets.
      \item Consider the closed form formula given by Theorem \ref{count}. The  formula implies that for a fixed $d$ and $k$, you can express the formula as a polynomial equation in terms of $q$. What is the degree of the polynomial for a given $d$ and $k$? Is it possible to generate the coefficients for the polynomial for a given $d$ and $k$ in closed form? What can this teach us about the count of rotational sets for a fixed degree and rotational set size?
    \item Consider the lattice \cite{Lattice} of rotational sets, partially ordered by subset inclusion, for a given degree $d$ and rotation number $\frac pq$, where a join \cite{Join} between sets $A$ and $B$ represents the union of $A$ and $B$.
    Given a degree $d$, what can we learn about the underlying structure of these lattices, and what can it teach us about rotational sets?
   \end{enumerate}

\bibliographystyle{amsplain}

\begin{thebibliography}{1}

\bibitem{Brittany:2023} B. Burdette, C. Falcione, C. Hale, and J. Mayer. \newblock Unicritcal and maximally critical laminations. \newblock 	arXiv:2303.17668 [math.DS] (2023).

\bibitem{Blokh:2006} A. Blokh, J. Malaugh, J. Mayer, L. Oversteegen, D. Parris. \newblock Rotational subsets of the circle under $z^d$. \newblock {\em Topology and its Applications} 153 (2006), 1540--1570.

\bibitem{Goldberg:1992} L. Goldberg. \newblock Fixed points of polynomial maps, Part 1, Rotation subsets of the circle. \newblock 
{\em Annales scientifiques de l’É.N.S. 4e série}, tome 25, no 6 (1992), p. 679--685.

\bibitem{Goldberg:1993} L. Goldberg and J. Milnor. \newblock Fixed points of polynomial maps, Part II.  Fixed point portraits. \newblock {\em Annales scientifiques de l’É.N.S. 4e série}, tome 26, no 1 (1993), p. 51--98.

\bibitem{McMullen:2010} C. McMullen. \newblock Dynamics on the unit disk: Short geodesics and simple cycles. \newblock {\em Commentarii Mathematici Helvetici}, 85 (2010) 723--749.

\bibitem{Tan:2019} Y. Tan. \newblock Counting rotational subsets of the circle $\mathbb{R}/\mathbb{Z}$ under the angle-multiplying map $t\mapsto dt$. \newblock arXiv:2207.03594v2 [math.CO] (2022).

\bibitem{Hugh:2023} Github: MaxSetGeneratingAlgo. https://github.com/mjmoorman03/MaxSetGeneratingAlgo/tree/main.

\bibitem{LamBuilder:2021} Falcione, C. Lamination Builder. https://csfalcione.github.io/lamination-builder/.

\bibitem{InEx} Wikipedia: Inclusion-Exclusion Principle. https://en.wikipedia.org/wiki/Inclusion-exclusion\_principle.

\bibitem{Join} Wikipedia: Join and Meet. https://en.wikipedia.org/wiki/Join\_and\_meet.

\bibitem{Lattice} Wikipedia: Lattice (order).
https://en.wikipedia.org/wiki/Lattice\_(order).

\bibitem{Mandel}
Wikipedia: Mandelbrot Set. https://en.wikipedia.org/wiki/Mandelbrot\_set.

\end{thebibliography}

\end{document}